\documentclass[reqno,a4paper]{amsart}
\usepackage[english,activeacute]{babel}
\usepackage{amssymb,amsmath,amsthm,amsfonts,mathrsfs,scalefnt,graphicx,graphics,color,hyperref,url,}
\usepackage{lmodern}
\usepackage{tikz}
\usetikzlibrary{trees}
\usetikzlibrary{shapes}

\theoremstyle{plain}
\newtheorem{theorem}{Theorem}[section]
\newtheorem{lemma}[theorem]{Lemma}
\newtheorem{proposition}[theorem]{Proposition}
\newtheorem{corollary}[theorem]{Corollary}
\newtheorem{definition}[theorem]{Definition}

\newtheorem{remark}[theorem]{Remark}
\numberwithin{equation}{section}

\newcommand{\Nb}  {{\mathbb N}}

\newcommand{\Zb}  {{\mathbb Z}}

\newcommand{\Fs} {{\mathcal F}}

\newcommand{\Ss} {{\mathcal S}}
\newcommand{\Vs} {{\mathcal V}}

\newcommand{\Ys} {{\mathcal Y}}

\DeclareMathOperator{\Var}{Var}

\renewcommand{\phi}{\varphi}

\newcommand{\De}{\Delta}
\newcommand{\la}{\lambda}

\newcommand{\ind}{1\!\kern-1pt \mathrm{I}}
\newcommand{\rsto}{]\!\kern-1.8pt ]}
\newcommand{\lsto}{[\!\kern-1.7pt [}

\newcommand{\cF}{\mathcal{F}}

\newcommand\F{\mbox{I\kern-2pt F}}

\title[The large fractional binary market]{Strong asymptotic arbitrage in the large fractional binary market}

\author{Fernando Cordero}
\address{Faculty of Technology, University of Bielefeld, Universit\"{a}tsstr. 25, 33615 Bielefeld, Germany}
\email{fcordero@techfak.uni-bielefeld.de}

\author{Lavinia Perez-Ostafe}
\address{Department of Mathematics, ETH Zurich, R\"{a}mistrasse 101, 8092 Zurich, Switzerland}
\email{lavinia.perez@math.ethz.ch}
\date{\today}%

\begin{document}


\begin{abstract}
 We study, from the perspective of large financial markets, the asymptotic arbitrage opportunities in a sequence of binary markets approximating the fractional Black-Scholes model. This approximating sequence was introduced by Sottinen and named fractional binary market. The large financial market under consideration does not satisfy the standard assumptions of the theory of asymptotic arbitrage. For this reason, we follow a constructive approach to show first that a strong type of asymptotic arbitrage exists in the large market without transaction costs. Indeed, with the help of an appropriate version of the law of large numbers and a stopping time procedure, we construct a sequence of self-financing trading strategies, which leads to the desired result. Next, we introduce, in each small market, proportional transaction costs, and we construct, following a similar argument, a sequence of self-financing trading strategies providing a strong asymptotic arbitrage when the transaction costs converge fast enough to $0$.
\end{abstract}
\maketitle
\section{Introduction}
The notion of large financial market was introduced by Kabanov and Kramkov in \cite{Kakra94} as a sequence of ordinary security market models. A suitable property for such kind of markets is the absence of asymptotic arbitrage opportunities. In the frictionless case, a standard assumption is that each small market is free of arbitrage. If, in addition, the small markets are complete, then the existence of asymptotic arbitrage opportunities is related to some contiguity properties of the sequence of equivalent martingale measures (see \cite{Kakra94}). These results are extended to the case of incomplete markets by Klein and Schachermayer in \cite{K:S:1996, K:Sch:1996} and by Kabanov and Kramkov in \cite{Kab:Kra:1998}. In the transaction costs case, i.e. when each small market is subject to proportional transaction costs, the standard assumption is that each small market is free of arbitrage under arbitrarily small transaction costs. In this context, characterizations of the existence of asymptotic arbitrage, similar to those in the frictionless case, can be found in \cite{Kl:Le:Pe}. 

In this paper, we consider a non-standard large financial market, i.e. a sequence of market models which admit arbitrage for sufficiently small transaction costs, and we study its asymptotic arbitrage opportunities. Our large financial market is given by the sequence of binary markets approximating the fractional Black-Scholes model introduced by Sottinen in \cite{Sotti}. We refer to this large financial market as the large fractional binary market. Similarly, we call $N$-fractional binary market to the $N$-period binary market in the sequence. Sottinen proves in \cite{Sotti} that, for $N$ sufficiently large, the $N$-fractional binary market admits arbitrage. From the results in \cite{CKO}, we conclude that the smallest transaction cost, $\lambda_c^N$, needed to eliminate the arbitrage in the $N$-fractional binary market is strictly positive. Moreover, from \cite{CP}, we know that $\lambda_c^N$ converges to $1$, a result which contrasts with the fact that the fractional Black-Scholes model is free of arbitrage under arbitrarily small transaction costs. This is not a true contradiction, since the arbitrage strategies constructed in \cite{CP} provide profits, under big transaction costs, with probabilities vanishing in the limit. As explained in \cite{CP}, a more appropriate way to compare the arbitrage opportunities in the sequence of fractional binary markets with the arbitrage opportunities in the fractional Black-Scholes market is to study the problem for the former from the perspective of the large financial markets. A first step in this direction was done in \cite{CP}, where the authors study the existence of asymptotic arbitrage of first kind (AA1) and of second kind (AA2) under the restriction of using only 1-step self-financing strategies. In this respect, it has been shown the existence of $1$-step AA1 in the large fractional binary market when the transaction costs are such that $\lambda_N=o(1/N^H)$. If, instead, $\lambda_N\sqrt{N}$ converges to infinity, then no 1-step asymptotic arbitrage of any kind appears in the model. Moreover, one can also show that, when the Hurst parameter $H$ is chosen close enough to $1/2$, then even in the frictionless large market there is no 1-step AA2.

In the present work, using more general self-financing trading strategies, we aim to construct, for an appropriate sequence of transaction costs, a strong asymptotic arbitrage, i.e. the possibility of getting arbitrarily rich with probability arbitrarily close to one while taking a vanishing risk. This problem can be viewed as a continuation of the study of asymptotic arbitrage initiated in \cite{CP}, in the sense that our trading strategies are chosen beyond the $1$-step setting of \cite{CP}. Not only that, the existence of this form of asymptotic arbitrage is stronger than AA1 and AA2 and, moreover, is obtained for any Hurst parameter $H>1/2.$

First, in the case of frictionless markets, we construct a candidate sequence of self-financing strategies, and we show that the value process of the portfolio can be expressed as a sum of dependent random variables. Due to this dependency, special versions of the law of large numbers are needed in order to conclude on the asymptotic behaviour of the value process at maturity. In this respect, with the help of a law of large numbers for mixingales (see \cite{An}), we prove that our strategies provide a strictly positive profit with probability strictly close to one. In order to construct a strong asymptotic arbitrage, we modify the sequence of trading strategies, first to ensure that the admissibility condition is satisfied and then to obtain an arbitrarily big profit. Indeed, using a well chosen sequence of stopping times, we stop the self-financing strategies at the first time the admissibility condition fails to hold. The resulting sequence of trading strategies paves the way to a strong asymptotic arbitrage. When transaction costs are taken into account, we show, following a similar argument, the existence of a strong asymptotic arbitrage when the transaction costs are of order $o(\sqrt{\log{N}}/N^{(2H-1/2)\wedge(H+1/2)})$. In direct comparison with the results of \cite{CP}, one can observe that, even if, when using a sequence of $1$-step self financing trading strategies, the rate of convergence of the transaction costs leading to an AA1 is better, this won't allow us though to obtain an AA2.

We emphasize that the methods presented in this work are not restricted to the chosen large financial market. To the contrary, since, in discrete time setting, the value process can be written as a sum of random variables, we believe that these techniques may be applicable also for other examples of discrete large financial markets. This is indeed the case whenever we dispose of an appropriate law of large numbers theorem and of a maximal inequality for the value process, in a similar manner as seen in our results.

The paper is structured as follows. In Section \ref{S2} we introduce the framework of our results, starting with the definition of a fractional binary market. We end this part with a short presentation of the concept of strong asymptotic arbitrage. In Section \ref{S3} and Section \ref{S4} the main results are concentrated. In Section \ref{S3}, we introduce a sequence of self-financing strategies leading to a strictly positive profit with probability close to one. A strong asymptotic arbitrage is then constructed using the aforementioned stopping procedure. In Section \ref{S4}, we extend this construction to the case when transaction costs are considered in the model.

\section{Preliminaries}\label{S2}

\subsection{Fractional binary markets}\label{bmao}
In this section, we briefly recall the so-called fractional binary markets, which were defined by Sottinen in \cite{Sotti} as a sequence of discrete markets approximating the fractional Black-Scholes model. 

First, we introduce the fractional Black-Scholes model. This continuous market takes the same form as the classical Black-Scholes model with the difference that the randomness of the risky asset is described by a fractional Brownian motion and not by a standard Brownian one. More precisely, the dynamics of the bond and of the stock are given by:
\begin{equation}\label{fbse}
 dB_t=r(t)\,B_t\, dt\quad\textrm{and}\quad dS_t=(a(t)+\sigma\,dZ_t)\, S_t,
\end{equation}
where $\sigma>0$ is a constant representing the volatility and $Z$ is a fractional  Brownian motion of Hurst parameter $H>1/2$. The functions $r$ and $a$ are deterministic and represent the interest rate and the drift of the stock. We assume in the sequel that the bond plays the role of the num\'eraire and, hence, $B_t=1$ for all times $t$ (i.e. $r=0$), and that $a$ is continuously differentiable.

Motivated by the construction of an easy example of arbitrage related to the fractional Black-Scholes model, Sottinen came up with the idea to express this special type of Black-Scholes model as a limiting process of a sequence of discrete markets with a binary structure. For this scope, he shows a Donsker-type theorem, in which the fractional Brownian motion is approximated by an inhomogeneous random walk. From this point on, he constructs a discrete model, called ``fractional binary market'', approximating \eqref{fbse}. Based on the results in \cite{CKP}, we provide here a simplified, but equivalent, presentation of these binary models. 

Let $(\Omega,\Fs, P)$ be a finite probability space and consider a sequence of i.i.d. random variables $(\xi_i)_{i\geq 1}$ such that
$P(\xi_1=-1)=P(\xi_1=1)=1/2.$
We denote $(\Fs_{i})_{i\geq 0}$ the induced filtration, i.e. $\Fs_{i}:=\sigma(\xi_1,\dots,\xi_i)$, for $i\geq1$, and $\Fs_0:=\{\emptyset,\Omega\}$. 

For each $N>1$, the $N$-fractional binary market is the discrete market in which the bond and stock are traded at the 
times $\{0, \frac{1}{N},...,\frac{N-1}{N}, 1\}$ under the dynamics:
\begin{equation}\label{fbm}
B_n^{N}=1\quad\textrm{and}\quad S_n^{N}=\left(1+a_n^{N}+\frac{X_n}{N^H}\right)\, S_{n-1}^{N}.
\end{equation}
We assume that the value of $S^{N}$ at time $0$ is constant, i.e.~$S_0^{N}=s_0$. The drift $a^{N}$ approximates the continuous drift given in \eqref{fbse} via $a_n^{N}=\frac{1}{N}a(n/N)$ and the process $(X_n)_{n\geq 1}$ can be expressed as
\begin{equation}\label{scale}
X_n:=\sum\limits_{i=1}^{n-1}j_n(i)\,\xi_i+g_n\xi_n,
\end{equation}
where
$$j_n(i):=\sigma\, C_H \int\limits_{i-1}^{i}x^{\frac{1}{2}-H}\left(\int\limits_0^1 (v+n-1)^{H-\frac{1}{2}} (v+n-1-x)^{H-\frac{3}{2}}dv\right) dx,$$
and
$$g_n:=\sigma\, C_H\int\limits_{n-1}^{n}x^{\frac{1}{2}-H}(n-x)^{H-\frac{1}{2}}\left(\int\limits_0^1 (y(n-x)+x)^{H-\frac{1}{2}}y^{H-\frac{3}{2}}dy\right)dx.$$

From \eqref{scale}, we see that $X_n$ is the sum of a process depending only on the information until time $n-1$ and a process depending only on the present. More precisely, $X_n=\Ys_n+g_n\xi_n$, where
$$\Ys_n:=\sum_{i=1}^{n-1}j_n(i)\xi_i.$$ 
Therefore, given the history up to time $n-1$, which fixes the values of $\Ys_n$ and $S_{n-1}^{N}$, the price process can take only two possible values at the next step:
$$\left(1+a_n^{N}+\frac{\Ys_{n}-g_n}{N^H}\right)S_{n-1}^{N}\quad\textrm{or}\quad\left(1+a_n^{N}+\frac{\Ys_{n}+g_n}{N^H}\right)S_{n-1}^{N}.$$
This brings to light the binary structure of these markets.

\subsubsection{Some useful estimations}\label{est}
We recall some estimations obtained in or easily derived from \cite{CKP} for the quantities involved in the definition of the fractional binary markets, i.e., $a_n^{N}$, $j_n$ and $g_n$.
 \begin{lemma}\label{eji}
There exist constants $C_1, C_2>0$ such that for all $i\geq 2$ we have
\begin{itemize}
 \item  For $1\leq\ell\leq i/4$:
$$j_i(\ell)\leq \frac{C_1}{i^{2-2H}\,\ell^{H-\frac{1}{2}}}.$$
\item For all $1\leq k\leq 3i/4$:
$$j_i(i-k)\leq \frac{C_2}{k^{\frac{3}{2}-H}}.$$
\end{itemize}
\end{lemma}
\begin{proof}
The proof of the first inequality follows using similar arguments to those used in the proof of \cite[Proposition 5.5]{CKP}. The second inequality uses analogous estimations to those obtained in the proof of \cite[Lemma 5.7]{CKP}.
\end{proof}
Next result corresponds to \cite[Lemma 3.5 and Theorem 5.10]{CKP}. 
\begin{lemma}\label{eg1}\ For all $1< n\leq N$, we have
 $$g\,\leq g_n\leq  g\,\left(1+\frac{1}{n-1}\right)^{H-\frac{1}{2}}\leq g\,2^{H-\frac{1}{2}},$$
 where $g:=\frac{\sigma c_H}{H+\frac12}$.
 This implies that $\lim_{n\rightarrow \infty}g_n=g$. Moreover,
 \begin{equation}\label{cys}
\Ys_n\xrightarrow[n\rightarrow\infty]{(d)}\Ys:= 2g\sum_{k=1}^{\infty}\rho(k)\xi_k,
\end{equation}
where $\rho$ denotes the autocovariance function of a fractional Brownian motion of Hurst parameter $h=\frac{H}2+\frac14\in(\frac12,\frac34)$, i.e.
$\rho(n):=\frac{1}{2}[(n+1)^{2h}+(n-1)^{2h}- 2n^{2h}]>0.$
 \end{lemma} 
 
For the drift term $a_n^{N}$, one can see using the definition and the continuity of the function $a$ that:
\begin{equation}\label{ea}
 |a_n^{(N)}|\leq \frac{{||a||}_\infty}{N},\qquad n\in\{1,...,N\}.
\end{equation}
This indicates that the contribution of $a_n^{N}$ is significantly less than the contribution of the other parameters of the model. Since we are interested in asymptotic properties of the fractional binary markets, the problem can be simplified by studying the case without the drift. Therefore, we assume henceforth that $a_n^{N}=0$ for all $1\leq n\leq N$.

\subsection{\texorpdfstring{Strong asymptotic arbitrage under transaction costs}{}}
It is well known that the fractional Black-Scholes model without friction is not free of arbitrage. This fact is also reflected in the approximating sequence of fractional binary markets, since, as shown by Sottinen in \cite{Sotti}, the $N$-fractional binary market admits, for $N$ sufficiently large, arbitrage opportunities. However, a pathological situation occurs when one introduces transaction costs. On the one hand, the fractional Black-Scholes model is free of arbitrage under arbitrarily small transaction costs. On the other hand, one can choose transaction costs $\lambda_N$ converging to $1$ such that the $N$-fractional binary market, for $N$ large enough, admits arbitrage under transaction costs $\lambda_N$ (see \cite{CP}). Despite this, the corresponding arbitrage opportunities disappear in the limit, in the sense that, the explicit strategies behind this counterintuitive behaviour provide strictly positive profits with probabilities vanishing in the limit. In order to avoid this kind of situations, we look here to the whole sequence of fractional markets as a large financial market, the large fractional binary market, and we study its asymptotic arbitrage opportunities, as introduced by Kabanov and Kramkov in \cite{Kakra94} and \cite{Kab:Kra:1998}. 

\begin{definition}[Large fractional binary market] The sequence of markets given by
 $\{(\Omega,\Fs,(\Fs_n)_{n=0}^N, P,S^{N})\}_{N\geq 1}$, where $S^{N}$ is the price process defined in \eqref{fbm}, is called large fractional binary market.
\end{definition}
We assume that the $N$-fractional binary market is subject to $\lambda_N\geq 0$ transaction costs ($\lambda_N=0$ corresponds to the frictionless case). We assume, without loss of generality, that we pay $\lambda_N$ transaction costs only when we sell and not when we buy. This means that the bid and ask price of the stock $S^{N}$ are modelled by the processes ${((1-\lambda)S^{N}_n)}_{n=0}^N$ and ${(S^{N}_n)}_{n=0}^N$ respectively.

\begin{definition}[$\lambda_N$-self-financing strategy]\label{sfs}
Given $\lambda_N\in[0,1]$, a $\lambda_N$-self-financing strategy for the process $S^{N}$ is an adapted process $\phi^N={(\phi_n^{0,N},\phi_n^{1,N})}_{n=-1}^N$ satisfying, for all $n\in\{0,...,N\}$, the following condition:
  \begin{equation}\label{selffin}
  \phi_n^{0,N}-\phi_{n-1}^{0,N}\leq -{(\phi_n^{1,N}-\phi_{n-1}^{1,N} )}^+\,S_n^{N}\, +\, (1-\lambda_N)\,{(\phi_n^{1,N}-\phi_{n-1}^{1,N} )}^-\,S_n^{N}.
 \end{equation}
Here $\phi^{0,N}$ denotes the number of units we hold in the bond and $\phi^{1,N}$ denotes the number of units in the stock. For such a $\lambda_N$-self-financing strategy, the liquidated value of the portfolio at each time $n$ is given by
$$V_n^{\la_N}(\phi^N):=\phi_n^{0,N}+(1-\la_N)(\phi_n^{1,N})^+S_n^{N}-(\phi_n^{1,N})^-S^{N}_n.$$ 
If $\la_N=0$, we simply write $V_n(\phi^N)$ instead of $V_n^0(\phi^N)$.
\end{definition}

\begin{remark}\label{csfs}
 For the purposes of this work, we can restrict our attention to self-financing strategies satisfying \eqref{selffin} with equality and having also that $\phi_N^{1,N}=0$. In other words, we avoid throwing away money and, at maturity, we liquidate the position in stock. For these kind of self-financing strategies, the values of $\phi_n^{0,N}$, $n\in\{0,...,N\}$, can be expressed in terms of the values of $\lambda_N$, $\phi_{-1}^{0,N}$ and $(\phi_k^{1,N})_{k=-1}^n$ as follows:
\begin{equation}\label{selffineq}
 \phi_n^{0,N}=  \phi_{-1}^{0,N}-\sum_{k=0}^n{(\De_k\phi^{1,N})}^+\,S^{N}_k\, +\, (1-\lambda_N)\,\sum_{k=0}^n{(\De_k\phi^{1,N})}^-\,S^{N}_k.
\end{equation}
In the previous identity, we use the notation $\De_n h:=h_n-h_{n-1}$. 

Equation \eqref{selffineq} gives us a way to construct self-financing strategies. More precisely, given $\lambda_N\geq 0$, a constant $\phi_{-1}^{0,N}$ and an adapted process $(\phi_k^{1,N})_{k=-1}^N$, we can use \eqref{selffineq} to define $(\phi_k^{0,N})_{k=0}^N$. The resulting adapted process ${(\phi_n^{0,N},\phi_n^{1,N})}_{n=-1}^N$ is by construction a $\lambda_N$-self-financing strategy, satisfying \eqref{selffin} with equality.
\end{remark}

In their work, Kabanov and Kramkov \cite{Kab:Kra:1998} distinguished between two kinds of asymptotic arbitrage, of the first kind and of the second kind. An asymptotic arbitrage of the first kind gives the possibility of getting arbitrarily rich with strictly positive probability by taking an arbitrarily small risk, whereas the second one is an opportunity of getting a strictly positive profit with probability arbitrarily close to $1$ by taking the risk of losing a uniformly bounded amount of money. The authors also considered a stronger version called ``strong asymptotic arbitrage'', which inherits the strong properties of the two mentioned kinds. More precisely, it can be seen as the possibility of getting arbitrarily rich with probability arbitrarily close to $1$ while taking a vanishing risk. We will work from now on with the latter concept.

We introduce now the definition of strong asymptotic arbitrage. For a detailed presentation on this topic, we refer the reader to \cite{Kab:Kra:1998} for frictionless markets and to \cite{Kl:Le:Pe} for markets with transaction costs.

\begin{definition}\label{SAAd}
There exists a strong asymptotic arbitrage (SAA) with transaction costs $\{\la_N\}_{N\geq 1}$ if there exists a subsequence of markets (again denoted by $N$) and self-financing trading strategies $\phi^N=(\phi^{0,N}, \phi^{1,N})$ with zero endowment for $S^{N}$ such that
\begin{enumerate}
\item{($c_N$-admissibility condition)} For all $i=0,\ldots,N$, $$V^{\la_N}_i(\phi^N)\geq-c_N,$$
\item $\lim_{N\to\infty}P^N(V_{N}^{\la_N}(\phi^N)\geq C_N)=1$
\end{enumerate}
where $c_N$ and $C_N$ are sequences of positive real numbers with $c_N\to0$ and $C_N\to\infty$.
\end{definition}
\begin{remark}
For self-financing strategies with zero endowment, and satisfying \eqref{selffin} with equality, the value process takes the following form:
\begin{align}\label{value}
V_n^{\lambda_N}(\phi^N)&=V_0^{\lambda_N}(\phi^N)+\sum\limits_{k=1}^{n}\phi_{k-1}^{1,N}\Delta_k S^{N}-\lambda_N\sum\limits_{k=1}^{n}\ind_{\{\Delta_k\phi^{1,N}\geq 0\}}\Delta_k\left[{(\phi^{1,N})}^+ S^{N}\right]\nonumber\\
&\qquad-\lambda_N\sum\limits_{k=1}^{n}\ind_{\{\Delta_k\phi^{1,N}< 0\}}\left\{\phi_{k-1}^{1,N}\Delta_k S^{N}+\Delta_k\left[{(\phi^{1,N})}^- S^{N}\right]\right\},
\end{align}
where
\begin{equation}\label{V0}
V^{\lambda_N}_0(\phi^N)=-\lambda_N|\phi^{1,N}_0|s_0.
\end{equation}
\end{remark}

\section{Strong asymptotic arbitrage without transaction costs}\label{S3}
As pointed out in the introduction, the large fractional binary market does not fulfil the standard conditions used in the theory of asymptotic arbitrage for large financial markets. For this reason, we use here a constructive approach to study the existence of strong asymptotic arbitrage. 
In this section, we consider the frictionless case. Our first goal is to construct self-financing strategies providing strictly positive profits at maturity with probability converging to one. To do so, we choose, for each $N\geq 1$, a trading strategy $\phi^N:=(\phi^{0,N},\phi^{1,N})$ similar to the one provided in \cite{BeSoVa} for the continuous framework. We have seen in Remark \ref{csfs} that, it is enough to indicate the position in stock $\phi^{1,N}$, as the position in bond $\phi^{0,N}$ can be derived from \eqref{selffineq}, setting $\lambda_N=0$ and $\phi_{-1}^{0,N}:=0$. Moreover, $\phi^{1,N}$ is chosen to be given by
$$\phi_{-1}^{1,N}:=\phi_0^{1,N}:=0\ \ \text{and}\ \ \phi_k^{1,N}:=N^{H-1}\frac{X_k}{S_k^{N}},\quad k\in\{1,\ldots,N\}.$$
Using \eqref{value} and \eqref{V0} with $\lambda_N=0$, we deduce that the value process associated to $\phi^N$ is given by
$$V_n(\phi^N)=\frac1{N}\sum_{k=1}^nX_{k-1}X_k,\quad n\in\{0,...,N\}.$$
Note that the terms in the sum can be expressed as
\begin{equation}\label{theta}
  X_{k-1}X_k=\theta_k^{(1)}+\theta_k^{(2)}+\theta_k^{(3)}+\theta_k^{(4)},
\end{equation}
where 
\begin{equation*}
\theta_k^{(1)}:=g_{k-1}g_k\xi_{k-1}\xi_k,\quad
\theta_k^{(2)}:=g_{k}\xi_k\Ys_{k-1},\quad
\theta_k^{(3)}:=g_{k-1}\xi_{k-1}\Ys_k,\quad
\theta_k^{(4)}:=\Ys_{k-1}\Ys_k.
\end{equation*}
Defining, for $i=1,2,3,4$,
$$\Ss_n^{(i)}=\sum_{k=1}^n\theta_{k}^{(i)},\quad n=1,\ldots,N,$$
we see that
\begin{equation}\label{VN}
V_n(\phi^{N})=\frac{1}{N}\left(\Ss_n^{(1)}+\Ss_n^{(2)}+\Ss_n^{(3)}+\Ss_n^{(4)}\right).
\end{equation}
\subsection{The value process at maturity and the law of large numbers}\label{s31}
We aim to characterize the asymptotic behaviour of $V_N(\phi^{N})$ by studying the asymptotic properties of each term entering \eqref{VN}. We will see that the first term in \eqref{VN} is a sum of pairwise independent random variables and hence, an appropriate extension of the law of large numbers to this situation can be applied (see \cite{CTT}). The asymptotic behaviour of the second and third terms in \eqref{VN} will be deduced by studying their variances. Finally, for the last term, we use a different approach, since the random variables $(\theta_k^{(4)})_{k\geq 1}$ are pairwise correlated. The notion of mixingale and a law of large numbers for uniformly integrable $L^1$-mixingales will play a key role in studying the behaviour of this term. For the ease of the reader, we recall now the notion of mixingale.
\begin{definition}[$L^p$-Mixingale]\label{mix}
 A sequence $\{X_k\}_{k\geq1}$ of random variables is an $L^p$-mixingale with respect to a given filtration $(\Fs_k)_{k\in\Zb}$, if there exist nonnegative constants $\{c_k\}_{k\geq1}$ and $\{\psi_m\}_{m\geq0}$ such that $\psi_m\to0$ as $m\to\infty$ and for all $k\geq1$ and $m\geq0$ the following hold:
 \begin{itemize}
  \item [(a)] ${\parallel E(X_k|\cF_{k-m}))\parallel}_p\leq c_k\psi_m$
  \item [(b)] ${\parallel X_k-E(X_k|\cF_{k+m}))\parallel}_p\leq c_k\psi_{m+1}$.
 \end{itemize}
\end{definition}

In order to study the last term in \eqref{VN}, we define:
 $$\Ys_k^*:=\Ys_{k-1}\Ys_k-E[\Ys_{k-1}\Ys_k],$$ and we consider the filtration $\mathbb{F}^*:=(\Fs_{i}^*)_{i\in\Zb}$ given by $\Fs_i^*:=\Fs_{i-1}$, for $i\geq 2$, and $\Fs_i^*:=\{\emptyset,\Omega\}$, for $i\leq 1$.
\begin{proposition}\label{mixi}
The process $(\Ys_k^*:k\geq 1)$ is an $L^2$-bounded $L^2$-mixingale with respect to $\mathbb{F}^*$. 
\end{proposition}
  \begin{proof}
  We first prove that the process $(\Ys_k^*:k\geq 1)$ is $L^2$-bounded. Note that
  $$E\left[|\Ys_{k-1}\Ys_k|^2\right]\leq E\left[|\Ys_{k-1}|^4\right]+E\left[|\Ys_k|^4\right],$$
  and using Khintchine's inequality (see \cite[(1)]{Khin}) for both terms on the right side, we obtain that
  $$E\left[|\Ys_{k-1}\Ys_k|^2\right]\leq 3\left(E\left[|\Ys_{k-1}|^2\right]^2+E\left[|\Ys_k|^2\right]^2\right).$$
  From \cite[Lemma 5.1]{CKP}, we conclude that $E[|\Ys_{k-1}\Ys_k|^2]$ is uniformly bounded, and therefore, $(\Ys_k^*:k\geq 1)$ is $L^2$-bounded.
  
  Now, we proceed to show that $\Ys_k^*$ is an $L^2$-mixingale with respect to $\mathbb{F}^*$, i.e. we have to check the two conditions of Definition~\ref{mix}. Note that, since $\Ys_k^*$ is $\Fs_k^*$-measurable, condition (b) is automatically satisfied. Hence, it remains to prove condition (a) of Definition \ref{mix}, i.e.
  \begin{equation}\label{mpa}
   {\parallel E[\Ys_k^*|\Fs_{k-m}^*]\parallel}_2\leq c_k\psi_m,\quad k\geq 1, m\geq 0,
  \end{equation}
  for some nonnegative constants $c_k$ and $\psi_m$ such that $\psi_m\to0$ as $m\to\infty$.
  
  Note that, for $k\leq m+1$, the left-hand side of \eqref{mpa} is equal to zero, and then, \eqref{mpa} holds for any choice of $c_k$ and $\psi_m$. The case $m=0$ can be easily treated using that $(\Ys_k^*:k\geq 1)$ is $L^2$-bounded. Now, we assume that $k-1>m\geq 1$, and we write $\Ys_{k-1}\Ys_k$ as follows:
  \begin{align}\label{y}
  \Ys_k\Ys_{k-1}&=\left(\sum_{l=1}^{k-1}j_k(l)\xi_l\right)\left(\sum_{l=1}^{k-2}j_{k-1}(l)\xi_l\right)\nonumber\\
  &=\left(\underbrace{\sum_{l=1}^{k-m-1}j_k(l)\xi_l}_{P_{k-m}^{(1)}}+\underbrace{\sum_{l=k-m}^{k-1}j_k(l)\xi_l}_{F_{k-m}^{(1)}}\right)\left(\underbrace{\sum_{l=1}^{k-m-1}j_{k-1}(l)\xi_l}_{P_{k-m}^{(2)}}+\underbrace{\sum_{l=k-m}^{k-2}j_{k-1}(l)\xi_l}_{F_{k-m}^{(2)}}\right)\nonumber\\
  &=P_{k-m}^{(1)}P_{k-m}^{(2)}+P_{k-m}^{(2)}F_{k-m}^{(1)}+P_{k-m}^{(1)}F_{k-m}^{(2)}+F_{k-m}^{(1)}F_{k-m}^{(2)},
  \end{align}
 Using that $P_{k-m}^{(i)}$ is independent of $F_{k-m}^{(j)}$, that $P_{k-m}^{(i)}$ is measurable with respect to $\Fs_{k-m}^*$ and that $F_{k-m}^{(i)}$ is independent of $\Fs_{k-m}^*$ for all $i,j\in\{1,2\}$, we deduce from \eqref{y} that
  \begin{equation}\label{exp}
  E[\Ys_k\Ys_{k-1}]=E\left[P_{k-m}^{(1)}P_{k-m}^{(2)}\right]+E\left[F_{k-m}^{(1)}F_{k-m}^{(2)}\right]
  \end{equation}
  and 
  \begin{equation}\label{condexp}
  E\left[\Ys_k\Ys_{k-1}|\Fs_{k-m}^*\right]=P_{k-m}^{(1)}P_{k-m}^{(2)}+E\left[F_{k-m}^{(1)}F_{k-m}^{(2)}\right].
  \end{equation}
  From \eqref{exp} and \eqref{condexp}, we have that
  \begin{align}\label{condexp2}
  E\left[\Ys_k^*|\Fs_{k-m}^*\right]=P_{k-m}^{(1)}P_{k-m}^{(2)}-E\left[P_{k-m}^{(1)}P_{k-m}^{(2)}\right].
  \end{align}
  Now, we write
  $$P_{k-m}^{(1)}P_{k-m}^{(2)}=\sum_{l\neq p}^{k-m-1}j_k(l)j_{k-1}(p)\xi_l\xi_p+\sum_{l=1}^{k-m-1}j_k(l)j_{k-1}(l),$$  
  which implies, using the independence of $\xi_l$ and $\xi_p$ for $l\neq p$, that
  $$E\left[P_{k-m}^{(1)}P_{k-m}^{(2)}\right]=\sum_{l=1}^{k-m-1}j_k(l)j_{k-1}(l),$$ 
  and hence
  $$E\left[\Ys_k^*|\Fs_{k-m}^*\right]=P_{k-m}^{(1)}P_{k-m}^{(2)}-E\left[P_{k-m}^{(1)}P_{k-m}^{(2)}\right]=\sum_{l\neq p}^{k-m-1}j_k(l)j_{k-1}(p)\xi_l\xi_p=:P_{k-m}^*.$$
Note first that
  \begin{equation}\label{pvar}
    E\left[|P_{k-m}^*|^2\right]\leq 2\sum_{l\neq p}^{k-m-1}(j_k(l)\,j_{k-1}(p))^2\leq 2\sum_{l=1}^{k-m-1}(j_k(l))^2\sum_{l=1}^{k-m-1}(j_{k-1}(l))^2.
  \end{equation}
  Additionally, using Lemma~\ref{eji}, we see that
    \begin{align*}
  \sum_{l=1}^{k-m-1}(j_k(l))^2&\leq\sum_{l=1}^{\frac{k}4}(j_k(l))^2+\sum_{l=m+1}^{\frac{3k}4-1}(j_k(k-l))^2\leq\sum_{l=1}^{\frac{k}4}(j_k(l))^2+\sum_{l=1}^{\frac{3k}4}(j_k(k-l))^2\\
&\leq\frac{C_1}{k^{4-4H}}\sum_{l=1}^{\frac{k}4}\frac1{l^{2H-1}}+C_2\sum_{l=1}^{\frac{3k}4}\frac1{l^{3-2H}}\leq \frac{C^0}{k^{2-2H}}\leq \frac{C^0}{m^{2-2H}},
  \end{align*}
  where $C^0>0$ is a well chosen constant. In the previous estimation, the term $\sum_{l=m+1}^{\frac{3k}4-1}(j_k(k-l))^2$ has to be understood as equal to zero if $k-m-1\leq k/4$.
  A similar argument shows that, there is a constant $C^*>0$ such that
  \begin{equation*}
   \sum_{l=1}^{k-m-1}(j_{k-1}(l))^2\leq \frac{C^*}{{m}^{2-2H}}.
  \end{equation*}
  Consequently, Equation \eqref{pvar} leads to
  \begin{equation*}
   E[|P_{k-m}^*|^2]\leq  \frac{C}{m^{4-4H}},
   \end{equation*}
   where $C>0$ is an appropriate constant. We therefore obtain that, for an appropriate constant $c>0$, the following holds:
 \begin{equation}
 \sqrt{ E[|E[\Ys_k^*|\Fs_{k-m}^*]|^2]}=\sqrt{E[|P_{k-m}^*|^2]}\leq c\,\frac1{m^{2-2H}}.
  \end{equation}
  The result follows by choosing $c_k:=c$ and $\psi_m:=m^{2H-2}$. 
  \end{proof}
  
\begin{remark}
 We have proved that $(\Ys_k^*,k\geq 1)$ is an $L^2$-bounded $L^2$-mixingale. In particular, $(\Ys_k^*,k\geq 1)$ is an uniformly integrable $L^1$-mixingale (we can use the same $c_\ell$ and $\psi_m$). Since in addition, $\sum_{k=1}^n c_k/n=c<\infty$, the conditions of the law of large numbers for mixingales given in \cite[Theorem 1]{An} are satisfied.
\end{remark}

The next result gives the asymptotic behaviour of  $V_N(\phi^{N})$ by studying the convergence properties of each term appearing in \eqref{VN}.

\begin{theorem}[Law of large numbers]\label{conv}
 The following statements hold:
 \begin{enumerate}
 \item $\frac1{N}\,\Ss_N^{(1)}\xrightarrow[N\to\infty]{a.s.}0$,\\
  \item $\frac1{N}\,\Ss_N^{(2)}\xrightarrow[N\to\infty]{L^2(P)}0$,\\
  \item $\frac1{N}\,\Ss_N^{(3)}\xrightarrow[N\to\infty]{L^2(P)}g^2(2^{H+\frac12}-2)>0$,
    \item $\frac1{N}\,\Ss_N^{(4)}\xrightarrow[N\to\infty]{L^1(P)}4g^2\sum\limits_{k=2}^\infty\rho(k)\rho(k-1)>0$.
 \end{enumerate} 
   In particular
  $$V_N(\phi^{N})\xrightarrow[N\to\infty]{P}\vartheta>0,$$
  where $\vartheta:=4g^2\sum_{k=2}^\infty\rho(k)\rho(k-1)+g^2(2^{H+\frac12}-2)$.
\end{theorem}

  \begin{proof}[Proof of Theorem \ref{conv}]
  
   (1) Note first that, for all $j\neq k$, we have
  $$P(\xi_k\xi_{k-1}=x|\xi_j\xi_{j-1}=y)=\frac12,$$
  independently of $y$, where $x,y\in\{-1,1\}$. Therefore, the random variables $(\theta_k^{(1)})_{k\geq 1}$ are pairwise independent. In addition, since $\Var[\xi_k\xi_{k-1}]=1$, we deduce from the estimates given in Lemma~\ref{eg1} that $\sum_{k=1}^N\frac1{k^2}\Var[\theta_k^{(1)}]<\infty$. Hence, the result follows as an application of the law of large numbers for pairwise independent random variables (see \cite[Theorem 1]{CTT}).
  
    (2) Note that $\xi_k$ is independent of $\Ys_{k-1}$, and in particular $E[\xi_k\Ys_{k-1}]=0$. Consequently, the convergence in $L^2(P)$ of $\Ss^{(2)}_N/N$ to $0$ is equivalent to the convergence of the variance to $0$. In addition, for any $k<j$, one can also see that 
  $$E[\xi_k\Ys_{k-1}\xi_j\Ys_{j-1}]=E[\xi_k\Ys_{k-1}\Ys_{j-1}]E[\xi_j]=0.$$
  It follows that
    \begin{align*}
  \Var\left[\frac1{N}\Ss^{(2)}_N\right]&=\frac1{N^2}\sum_{k=1}^N{g_k}^2\Var[\xi_k\Ys_{k-1}]=\frac1{N^2}\sum_{k=1}^N{g_k}^2E[\xi_k^2(\Ys_{k-1})^2]\\
  &=\frac1{N^2}\sum_{k=1}^N {g_k}^2E[(\Ys_{k-1})^2]=\frac1{N^2}\sum_{k=1}^N {g_k}^2\Var[\Ys_{k-1}].
  \end{align*}
  By Lemma~\ref{eg1}, we have that $g_k\leq 2^{H-\frac12}g$ and by \cite[Lemma 5.1]{CKP}, we have that $$\sup_{k}\Var[\Ys_{k-1}]<\infty.$$ Consequently, we obtain
  \begin{align*}
  \Var\left[\frac1{N}\Ss^{(2)}_N\right]&\leq\frac{M}{N}\xrightarrow[N\to\infty]{}0
  \end{align*}
  for some constant $M>0$. This gives us the convergence of $\Ss_N^{(2)}$ to $0$ in $L^2(P)$.\\ 
  
  (3) We write $\Ss_N^{(3)}$ as a sum of a random term and a deterministic one. We prove that the variance of the random term converges to $0$ and the deterministic term converges to $g^2(2^{H+\frac12}-2)>0$. Indeed, denoting $\tilde{\Ys}_{k-1}:=\sum_{l=1}^{k-2}j_k(l)\xi_l$, we have that
  \begin{align*}
  \frac1{N}\Ss_N^{(3)}&=\frac1{N}\sum_{k=1}^Ng_{k-1}\xi_{k-1}\Ys_k=\frac1{N}\sum_{k=1}^Ng_{k-1}\xi_{k-1}\tilde{\Ys}_{k-1}+\frac1{N}\sum_{k=1}^Ng_{k-1}j_k(k-1).
  \end{align*}
  From \eqref{eg1} and \cite[Eq. 5.3]{CKP}, we see that $g_{k-1}j_k(k-1)\xrightarrow[k\to\infty]{}g^2(2^{H+\frac12}-2)$. As a consequence, we deduce that
  $$\frac1{N}\sum_{k=1}^Ng_{k-1}j_k(k-1)\xrightarrow[N\to\infty]{}g^2(2^{H+\frac12}-2).$$
  For the random term, since $\tilde{\Ys}_{k-1}$ is independent of $\xi_{k-1}$, and using a similar argument like 
 in the previous part, one obtains that
 $$\frac1{N}\sum_{k=1}^Ng_{k-1}\xi_{k-1}\tilde{\Ys}_{k-1}\xrightarrow[N\to\infty]{L^2(P)}0$$
 and hence the desired result.\\
 
 (4) Directly from Proposition~\ref{mixi} and the law of large numbers for uniformly integrable $L^1$-mixingales (see \cite[Theorem 1]{An}). Indeed, by the two mentioned results we have that
 \begin{align}\label{a}
 \frac1{N}\sum_{k=1}^N\left(\Ys_{k-1}\Ys_k-E[\Ys_{k-1}\Ys_k]\right)\xrightarrow[N\to\infty]{L^1(P)}0.
 \end{align}
 In addition, for $n\geq 4$, we have 
 $$E[\Ys_{n-1}\Ys_n]=\sum\limits_{i=1}^{\frac{n}{4}}j_n(i)j_{n-1}(i)+\sum\limits_{i=\frac{n}{4}+1}^{n-2}j_n(i)j_{n-1}(i).$$
 Using the estimates given in Lemma \ref{eji}, we deduce that the first sum on the right-hand side converges to zero. For the second sum, following 
 the lines of the proof of \cite[Lemma 5.7]{CKP}, we obtain that 
 $$\sum\limits_{i=\frac{n}{4}+1}^{n-2}j_n(i)j_{n-1}(i)=\sum\limits_{k=2}^{\frac{3n}{4}-1}j_n(n-k)j_{n-1}(n-k)\xrightarrow[n\to\infty]{}4g^2\sum\limits_{k=2}^\infty\rho(k)\rho(k-1):=\Vs>0.$$
Consequently, we conclude that
 $$E[\Ys_{n-1}\Ys_n]\xrightarrow[n\to\infty]{}\Vs,$$
and therefore
\begin{equation}\label{cyy}
 \frac1{N}\sum_{k=1}^NE[\Ys_{k-1}\Ys_k]\xrightarrow[N\to\infty]{}\Vs.
\end{equation}
Thus, we have
 \begin{align}\label{b}
 E\left[\left|\frac1{N}\sum_{k=1}^N\Ys_{k-1}\Ys_k-\Vs\right|\right] &\leq E\left[\left|\frac1{N}\sum_{k=1}^N\left(\Ys_{k-1}\Ys_k-E[\Ys_{k-1}\Ys_k]\right)\right|\right]\nonumber\\ 
&\quad +\left|\frac1{N}\sum_{k=1}^N E[\Ys_{k-1}\Ys_k]-\Vs \right|.
\end{align}
By \eqref{a}, \eqref{cyy} and \eqref{b}, it follows immediately that
$$E\left[\left|\frac1{N}\sum_{k=1}^N\Ys_{k-1}\Ys_k-\Vs\right|\right] \xrightarrow[N\to\infty]{}0,$$
and hence
$$\frac1{N}\Ss_N^{(4)}\xrightarrow[N\to\infty]{L^1(P)}\Vs.$$
The proof of (4) is now complete.
\end{proof}

\begin{corollary}\label{AA1conv}
For all $\varepsilon>0$,
 $$P(V_N(\phi^{N})>\vartheta(1-\varepsilon))\xrightarrow[N\to\infty]{}1.$$
\end{corollary}
\begin{proof}
 The result follows using Theorem~\ref{conv} and the definition of the convergence in probability.
\end{proof}

\subsection{Admissibility condition through stopping procedure}\label{s32}
The sequence of self-financing strategies $(\phi^N)_{N\geq 1}$ constructed in the previous section gives the possibility to make a strictly positive profit with probability arbitrarily close to one. Now, we proceed to modify our strategies in such a way that the admissibility conditions are satisfied. More precisely, we will stop our self-financing strategies at the first time they fail the admissibility condition. To do so, we split the value process as in \eqref{VN}, and we study the stopping times corresponding to each part.

For each $i=1,2,3,4$ and any sequence of strictly positive numbers $(\varepsilon_N,N\geq 1)$, we define the stopping time
$$T_{\varepsilon_N}^{(N,i)}:=\inf\{k\in\{1,\ldots,N\}:\, \frac{1}{N}\,\Ss_k^{(i)}< -\varepsilon_N\},$$
with the convention that $\inf\emptyset =\infty$. Note that these stopping times have values on $\{1,\dots,N\}\cup\{\infty\}$.

We study the first three stopping times with the help of an extension of the Kolmogorov's maximal inequality. A different approach is used in the study of $T_{\varepsilon_N}^{(N,4)}$.
\subsubsection{Maximal inequalities for the first three terms}
We start with the following generalization of the Kolmogorov's maximal inequality, which fits our setting.
Let $c$ be a strictly positive constant and $W=(W_k)_{k=1}^N$ a sequence of centred random variables. We define the stopping time 
$$T_c(W):=\inf\{k\in\{1,\ldots,N\}:\, |W_k|> c\}.$$
\begin{lemma}[Maximal inequality]\label{mi}
Assume that, for all $k\in\{1,\ldots,N\}$, we have
\begin{equation}\label{cmi}
 E\left[(W_N-W_k)\,W_k\, 1_{\{T_c(W)=k\}}\right]=0.
\end{equation}
Then
$$P\left(\sup\limits_{1\leq i\leq N}|W_i|> c\right)=P(T_c(W)\leq N)\leq \frac{\Var(W_N)}{c^2}.$$
\end{lemma}
\begin{proof}
We may assume that $\Var(W_N)<\infty$. Note that
\begin{align*}
\Var(W_N)&=E[W_N^2]\geq \sum\limits_{k=1}^{N} E[(W_k +W_N-W_k)^2 1_{\{T_c(W)=k\}}]\\
&\geq \sum\limits_{k=1}^{N} \left(E[W_k^2 1_{\{T_c(W)=k\}}]+ 2E\left[(W_N-W_k)W_k 1_{\{T_c(W)=k\}}\right]\right)\\
&=\sum\limits_{k=1}^{N} E[W_k^2 1_{\{T_c(W)=k\}}]\geq c^2P(T_c(W)\leq N).
\end{align*}
The result follows.
\end{proof}
The next result will be obtained as an application of the previous lemma.
\begin{lemma}\label{mi123}
For each $i=1,2,3$, there is a constant $C^{(i)}>0$, such that
$$P\left(T_{\varepsilon_N}^{(N,i)}\leq N\right)\leq \frac{C^{(i)}}{N\,\varepsilon_N^2}.$$
\end{lemma}
\begin{proof}
(1) Define the stopping time $$\widetilde{T}_{\varepsilon_N}^{(N,1)}:=T_{\varepsilon_N}\left(\frac{1}{N}\Ss^{(1)}\right),$$
and note that $T_{\varepsilon_N}^{(N,1)}\geq\widetilde{T}_{\varepsilon_N}^{(N,1)}$. Therefore, it is enough to prove the result for $\widetilde{T}_{\varepsilon_N}^{(N,1)}$.
 Since $g_k\leq 2g$ and the random variables $\{\xi_{k-1} \xi_k; k\geq 1\}$ are pairwise independent, we conclude that 
 $$\Var\left(\frac{1}{N}\Ss_N^{(1)}\right)=\frac{1}{N^2}\sum\limits_{\ell=1}^{N}g_{\ell-1}^2 g_\ell^2\leq \frac{(2g)^4}{N}.$$
Note that $\Ss_k^{(1)}1_{\{\widetilde{T}_{\varepsilon_N}^{(N,1)}=k\}}$ is $\sigma(\xi_1,...,\xi_k)$-measurable. In addition, we have that
$$\Ss_N^{(1)}-\Ss_k^{(1)}=\sum\limits_{\ell=k+2}^{N}g_{\ell-1} g_\ell \xi_{\ell-1} \xi_\ell+g_{k} g_{k+1} \xi_{k} \xi_{k+1}.$$
Moreover, $\sum_{\ell=k+2}^{N}g_{\ell-1} g_\ell\, \xi_{\ell-1} \xi_\ell$ is $\sigma(\xi_{k+1},...,\xi_N)$-measurable and 
$$E\left[\xi_{k} \xi_{k+1}\,\Ss_k^{(1)}1_{\{\widetilde{T}_{\varepsilon_N}^{(N,1)}=k\}}\right]=E\left[\xi_{k+1}\right]E\left[ \xi_k\,\Ss_k^{(1)}1_{\{\widetilde{T}_{\varepsilon_N}^{(N,1)}=k\}}\right]=0.$$
Consequently, the condition \eqref{cmi} is satisfied and the result follows as an application of Lemma \ref{mi}.

(2) Define the stopping time $$\widetilde{T}_{\varepsilon_N}^{(N,2)}:=T_{\varepsilon_N}\left(\frac{1}{N}\Ss^{(2)}\right),$$
and note that $T_{\varepsilon_N}^{(N,2)}\geq\widetilde{T}_{\varepsilon_N}^{(N,2)}$. As before, we conclude that it is enough to prove the result for $\widetilde{T}_{\varepsilon_N}^{(N,2)}$.

 Since, for $k>j$, $E[\xi_k\Ys_{k-1}\xi_j\Ys_{j-1}]=E[\xi_k]\,E[\Ys_{k-1}\xi_j\Ys_{j-1}]=0$, we have that
 $$\Var\left(\frac{1}{N}\Ss_N^{(2)}\right)=\frac{1}{N^2}\sum\limits_{\ell=1}^{N}\sum\limits_{i=1}^{\ell-2} g_\ell^2 j_{\ell-1}^2(i)\leq \frac{(2g)^2}{N^2}\sum\limits_{\ell=1}^{N}\sum\limits_{i=1}^{\ell-2} j_{\ell-1}^2(i).$$
 From \cite[Lemma 5.1]{CKP}, the quantities $\sum_{i=1}^{\ell-2} j_{\ell-1}^2(i)$ are uniformly bounded. We conclude that there is $C^{(2)}>0$ such that
 $$\Var\left(\frac{1}{N}\Ss_N^{(2)}\right)\leq\frac{C^{(2)}}{N}.$$
 Additionally, we have that
 $$\Ss_N^{(2)}-\Ss_k^{(2)}=\sum\limits_{\ell=k+1}^{N} g_\ell\xi_\ell\Ys_{\ell-1}.$$
 On the other hand, for all $\ell\in\{k+1,...,N\}$, we have that
 $$E\left[\xi_{\ell} \Ys_{\ell-1}\,\Ss_k^{(2)}1_{\{\widetilde{T}_{\varepsilon_N}^{(N,2)}=k\}}\right]=E\left[\xi_{\ell}\right]E\left[ \Ys_{\ell-1}\,\Ss_k^{(2)}1_{\{\widetilde{T}_{\varepsilon_N}^{(N,2)}=k\}}\right]=0.$$
The condition \eqref{cmi} is verified and the result follows.

(3) Denote $\tilde{\Ys}_{k-1}:=\sum_{l=1}^{k-2}j_k(l)\xi_l$ and note that:
$$\theta_k^{(3)}=g_{k-1} j_k(k-1) + \underbrace{g_{k-1}\xi_{k-1}\tilde{\Ys}_{k-1}}_{=:\tilde{\theta}_k^{(3)}}>\tilde{\theta}_k^{(3)}.$$
As a consequence, for each $n\in\{1,\ldots ,N\}$, we have
$$\Ss_n^{(3)}>\sum\limits_{k=1}^n \tilde{\theta}_k^{(3)}=:\tilde{\Ss}_n^{(3)}.$$
Moreover, if we define
 $$\widetilde{T}_{\varepsilon_N}^{(N,3)}:=T_{\varepsilon_N}\left(\frac{1}{N}\tilde{\Ss}^{(3)}\right),$$
 it follows that $T_{\varepsilon_N}^{(N,3)}\geq \widetilde{T}_{\varepsilon_N}^{(N,3)}.$ Thus, it will be enough to prove the result for $\widetilde{T}_{\varepsilon_N}^{(N,3)}$.
 
Since, for $k>j$, $E[\xi_{k-1}\tilde{\Ys}_{k-1}\xi_{j-1}\tilde{\Ys}_{j-1}]=E[\xi_{k-1}]\,E[\tilde{\Ys}_{k-1}\xi_{j-1}\tilde{\Ys}_{j-1}]=0$, we get
 $$\Var\left(\frac{1}{N}\tilde{\Ss}_N^{(3)}\right)=\frac{1}{N^2}\sum\limits_{\ell=1}^{N}\sum\limits_{i=1}^{\ell-2} g_{\ell-1}^2 j_{\ell}^2(i)\leq \frac{(2g)^2}{N^2}\sum\limits_{\ell=1}^{N}\sum\limits_{i=1}^{\ell-2} j_{\ell}^2(i).$$
 We know from \cite[Lemma 5.1]{CKP} that the quantities $\sum_{i=1}^{\ell-2} j_{\ell}^2(i)$ are uniformly bounded. We conclude that there is $C^{(3)}>0$ such that
 $$\Var\left(\frac{1}{N}\tilde{\Ss}_N^{(3)}\right)\leq\frac{C^{(3)}}{N}.$$
 Additionally, we have
 $$\tilde{\Ss}_N^{(3)}-\tilde{\Ss}_k^{(3)}=\sum\limits_{\ell=k+1}^{N} g_{\ell-1}\xi_{\ell-1}\tilde{\Ys}_{\ell-1}.$$
 On the other hand, for all $\ell\in\{k+1,...,N\}$, we have that
 $$E\left[\xi_{\ell-1} \tilde{\Ys}_{\ell-1}\,\tilde{\Ss}_k^{(3)}1_{\{\widetilde{T}_{\varepsilon_N}^{(N,3)}=k\}}\right]=E\left[\xi_{\ell-1}\right]E\left[ \tilde{\Ys}_{\ell-1}\,\tilde{\Ss}_k^{(3)}1_{\{\widetilde{T}_{\varepsilon_N}^{(N,3)}=k\}}\right]=0.$$
The condition in Lemma \ref{mi} is verified and the result follows.
 \end{proof}
 
\subsubsection{Maximal inequality for the last term}
In the previous three cases, we consider the process $W^{(i)}=\Ss^{(i)}-E[\Ss^{(i)}]$, and we use either a pairwise independence argument or the orthogonality of some random variables to prove that condition \eqref{cmi} is satisfied. The desired results are obtained with the help of Lemma \ref{mi}. For the stopping time $T_{\varepsilon_N}^{(N,4)}$, we can not proceed in the same way, because the random variables $(\Ys_k^*)_{k\geq 1}$ are pairwise correlated. Nevertheless, the key ingredient is again a maximal inequality, this time for the process $(\Ys_k^*)_{k\geq 1}$. 

Let's define the random variables
$$X_{i,k}:=E[\Ys_i^*|\Fs_{i-k}^*]-E[\Ys_i^*|\Fs_{i-k-1}^*],\quad i\in\Nb,\ k\in \Zb.$$
Note that $X_{i,k}=0$ if $k<0$ or $i\leq k+1$. As a consequence, we have that:
$$Y_{n,k}:=\sum\limits_{i=1}^n X_{i,k}=\sum\limits_{i=k+2}^n X_{i,k}.$$
We also denote
$$\Ss_n^*:=\sum\limits_{k=1}^n \Ys_k^*= \Ss_n^{(4)}-E[\Ss_n^{(4)}].$$
The following result provides the desired maximal inequality for $(\Ss_n^*)_{n\geq 1}$. 
\begin{lemma}\label{lmilm}
For all $n\geq 1$, we have that
$$\Ss_n^*=\sum\limits_{k=-\infty}^\infty Y_{n,k}=\sum\limits_{k=0}^{n-2} Y_{n,k}\quad \textrm{a.s.},$$
and for any sequence of strictly positive numbers $(a_k:k\in\Zb)$, we have
\begin{equation}\label{milm}
 E\left[\sup_{n\leq N}|\Ss_n^*|^2\right]\leq 4\left(\sum\limits_{k=0}^{N-2} a_k\right)\left(\sum\limits_{k=0}^{N-2} a_k^{-1}\Var(Y_{N,k})\right).
\end{equation}
\end{lemma}
\begin{proof}
 From Proposition \ref{mixi}, we deduce that $(\Ys^*_n)_{n\geq 1}$ is a sequence of square integrable random variables with zero mean. It is also straightforward to see that
$$E[\Ys_n^*|\Fs^*_{-\infty}]=\Ys_n^*-E[\Ys_n^*|\Fs^*_{\infty}]=0\quad \textrm{a.s.},$$
where $\Fs^*_{-\infty}:=\{\emptyset,\Omega\}$ and $\Fs^*_{\infty}:=\sigma(\xi_i:i\geq 1)$. Therefore, the first statement follows as a direct application of \cite[Lemma 1.5]{ML} and the fact that $Y_{n,k}=0$ for $k<0$ and $k>n-2$. For the remaining part, we need to slightly modify the arguments of \cite[Lemma 1.5]{ML}.
First note that, for $k\geq 1$, $(Y_{n,k})_{n\geq 1}$ is a square integrable $(\Fs^*_{n-k})_{n\geq 1}$-martingale. On the other hand, using Cauchy-Schwartz
$$(\Ss_n^*)^2=\left(\sum\limits_{k=0}^{n-2} \sqrt{a_k}\,\frac{Y_{n,k}}{\sqrt{a_k}}\right)^2\leq \left(\sum\limits_{k=0}^{n-2}a_k\right)\left(\sum\limits_{k=0}^{n-2} \frac{Y_{n,k}^2}{a_k}\right)\leq \left(\sum\limits_{k=0}^{N-2}a_k\right)\left(\sum\limits_{k=0}^{N-2} \frac{Y_{n,k}^2}{a_k}\right).$$
Now, we take $\sup_{n\leq N}$ in both extremes of this inequality, then we take expectations, and we apply Doob's inequality to bound the right-hand side. The result follows.
\end{proof}
In order to obtain an explicit upper bound for the left-hand side in \eqref{milm}, we start by studying the variance of $Y_{N,k}$.
\begin{lemma}\label{rxni}
For all $0\leq k< i-1$ we have that
$$X_{i,k}=\xi_{i-k-1}\sum\limits_{\ell=1}^{i-k-2}j_i(\ell,i-k-1)\xi_\ell,$$
where $j_i(\ell,p):=j_i(\ell)j_{i-1}(p)+j_i(p)j_{i-1}(\ell).$ In particular, for each $k\geq 0$, the random variables $(X_{i,k}: i> k+1)$ are centred and pairwise uncorrelated. As a consequence, we have that 
$$\Var(Y_{N,k})=\sum\limits_{i=k+2}^N\sum\limits_{\ell=1}^{i-k-2}j_i(\ell,i-k-1)^2.$$
\end{lemma}
\begin{proof}
It is straightforward from the expression for $E[\Ys_i^*|\Fs_{i-k}^*]$ obtained in the proof of Proposition \ref{mixi}.
\end{proof}
Next result gives an explicit upper bound for the left-hand side in \eqref{milm}.
\begin{lemma}\label{mub}
There is a constant $C^*>0$ such that 
$$E\left[\sup_{n\leq N}|\Ss_n^*|^2\right]\leq C^* \ln(N)N^{4H-2}.$$
\end{lemma}
\begin{proof}
We note first that
$$j_i(\ell,p)^2\leq 2\left( j_i(\ell)^2j_{i-1}(p)^2+j_i(p)^2j_{i-1}(\ell)^2\right).$$
Using Lemma \ref{rxni}, we see that
$$\frac{\Var(Y_{N,k})}{2}\leq \underbrace{\sum\limits_{i=k+2}^N\sum\limits_{\ell=1}^{i-k-2}j_i(i-k-1)^2j_{i-1}(\ell)^2}_{:=V_{N,k}}+\underbrace{\sum\limits_{i=k+2}^N\sum\limits_{\ell=1}^{i-k-2}j_i(\ell)^2j_{i-1}(i-k-1)^2}_{:=W_{N,k}}.$$
Now, we write $V_{N,k}=V_{N,k}^1+V_{N,k}^2+V_{N,k}^3$, where
\begin{align*}
V_{N,k}^1&:=\sum\limits_{i=k+2}^{\frac{4(k+1)}{3}\wedge N}j_i(i-k-1)^2\sum\limits_{\ell=1}^{i-k-2}j_{i-1}(\ell)^2,\\
V_{N,k}^2&:=\sum\limits_{i=\frac{4(k+1)}{3}\wedge N+1}^{N}j_i(i-k-1)^2\sum\limits_{\ell=1}^{\frac{i-1}{4}}j_{i-1}(\ell)^2,\\
V_{N,k}^3&:=\sum\limits_{i=\frac{4(k+1)}{3}\wedge N+1}^{N}j_i(i-k-1)^2\sum\limits_{\ell=\frac{i+3}{4}}^{i-k-2}j_{i-1}(\ell)^2.
\end{align*}
For $V_{N,k}^1$, we use Lemma \ref{eji} to obtain, that
\begin{align*}
 V_{N,k}^1&\leq C_1^4\sum\limits_{i=k+2}^{\frac{4(k+1)}{3}}\frac{1}{(i-1)^{8-8H}(i-k-1)^{2H-1}}\sum\limits_{\ell=1}^{i-k-2}\frac{1}{\ell^{2H-1}}\\
 &\leq \frac{{C_1^4}}{(2-2H)}\sum\limits_{i=k+2}^{\frac{4(k+1)}{3}}\frac{(i-k-2)^{2-2H}}{(i-1)^{8-8H}(i-k-1)^{2H-1}}\\
 &\leq \frac{{C_1^4}}{(2-2H)(k+1)^{6-6H}}\sum\limits_{i=k+2}^{\frac{4(k+1)}{3}}\frac{1}{(i-k-1)^{2H-1}}\leq \frac{\widehat{C}_1}{(k+1)^{4-4H}},
\end{align*}
where $\widehat{C}_1>0$ is an appropriate constant. For the other terms, we assume that $\frac{4(k+1)}{3}\leq N$, otherwise they are trivially equal to zero. Thus, for  $V_{N,k}^2$, we have
\begin{align*}
 V_{N,k}^2&\leq \frac{(C_1\, C_2)^2}{(k+1)^{3-2H}}\sum\limits_{i=\frac{4(k+1)}{3}+1}^{N}\frac{1}{(i-1)^{4-4H}}\sum\limits_{\ell=1}^{\frac{i-1}{4}}\frac{1}{\ell^{2H-1}}\\
 &\leq \frac{(C_1\, C_2)^2}{(2-2H)(k+1)^{3-2H}}\sum\limits_{i=2}^{N}\frac{1}{(i-1)^{2-2H}}\leq \frac{\widehat{C}_2 N^{2H-1}}{(k+1)^{3-2H}},\\
 \end{align*}
where $\widehat{C}_2>0$ is a well chosen constant. Similarly, for the last term we have
\begin{align*}
 V_{N,k}^3&\leq \frac{C_2^2}{(k+1)^{3-2H}}\sum\limits_{i=\frac{4(k+1)}{3}+1}^{N}\sum\limits_{\ell=k+1}^{\frac{3(i-1)}{4}}j_{i-1}(i-1-\ell)^2\leq \frac{\widehat{C}_3 N}{(k+1)^{5-4H}},
 \end{align*}
 where $\widehat{C}_3>0$ is a well chosen constant. We conclude that, there exists $C_0>0$ such that 
 $$V_{N,k}\leq\frac{C_0 N}{(k+1)^{5-4H}}.$$
for all $k\leq N-2$. An upper bound of the same order for $W_{N,k}$ can be obtained using similar arguments. Therefore, there is $C^*>0$, such that
 $$\Var(Y_{N,k})\leq\frac{C^* N}{(k+1)^{5-4H}}.$$
The result follows by plugging this upper bound in \eqref{milm} with $a_k:=(k+1)^{-1}$. 
\end{proof}
As a consequence of the previous results, we obtain the following analogue of Lemma \ref{mi123} for $T_{\varepsilon_N}^{(N,4)}$.
\begin{corollary}\label{T4}
There is a constant $C^{(4)}>0$ such that
$$P\left(T_{\varepsilon_N}^{(N,4)}\leq N\right)\leq \frac{C^{(4)}\ln(N)}{N^{4-4H}\varepsilon_N^2}.$$
\end{corollary}
\begin{proof}
First, note that, 
$$E\left[\Ss_n^{(4)}\right]=\sum_{k=3}^{n}\sum_{i=1}^{k-2}j_{k}(i)j_{k-1}(i)\geq 0.$$
Therefore, we have that
$$\Ss_n^{(4)}=\Ss_n^*+E[\Ss_n^{(4)}]\geq \Ss_n^*.$$
Consequently, if we define the stopping time $T_{\varepsilon_N}^{(N,*)}$ as follows
$$T_{\varepsilon_N}^{(N,*)}:=\inf\left\{k\in\{1,\ldots,N\}:\, \frac{1}{N} |\Ss_k^{*}|> \varepsilon_N\right\},$$
then $T_{\varepsilon_N}^{(N,4)}\geq T_{\varepsilon_N}^{(N,*)}$. In particular, we deduce that 
$$P\left(T_{\varepsilon_N}^{(N,4)}\leq N\right)\leq P\left(T_{\varepsilon_N}^{(N,*)}\leq N\right)=P\left(\sup_{n\leq N}\frac1{N}|\Ss_n^*|> \varepsilon_N\right).$$
The result follows as an application of the Tchebychev inequality and Lemma \ref{mub}.
\end{proof}
\subsection{A strong asymptotic arbitrage}
In this section, using the results of section \ref{s32}, we modify the sequence $(\phi^N)_{N\geq 1}$ constructed in section \ref{s31}, in order to construct an explicit strong asymptotic arbitrage. A first modification will lead to a sequence of self-financing strategies $(\hat{\phi}^N)_{N\geq 1}$ providing a strictly positive profit with probability arbitrarily close to one and verifying the admissibility conditions. Finally, after a second modification, we will obtain a new sequence of self-financing strategies $(\psi^N)_{N\geq 1}$ leading to the desired strong asymptotic arbitrage.

The sequence $(\hat{\phi}^N)_{N\geq 1}$ is defined as follows. The position in stock is given by
$$\hat{\phi}_k^{1,N}:=1_{\{k<T_N\}}\phi_k^{1,N},\quad k\in\{-1,0,...,N\},$$ where 
\begin{equation}\label{stoptime}
 T_N:=T_{\varepsilon_N}^{(N,1)}\wedge T_{\varepsilon_N}^{(N,2)}\wedge \left(T_{\varepsilon_N}^{(N,3)}-1\right)\wedge \left(T_{\varepsilon_N}^{(N,4)}-1\right),
\end{equation}
and the position in bond is derived from \eqref{selffineq} setting $\lambda_N=0$ and $\hat{\phi}_{-1}^{0,N}=0$. Note that, since the random variables $\Ss_n^{(3)}$ and $\Ss_n^{(4)}$ are $\Fs_{n-1}$-measurable, $T_{\varepsilon_N}^{(N,3)}-1$ and $T_{\varepsilon_N}^{(N,4)}-1$ are stopping times with respect to $(\Fs_n)_{n=0}^N$. Clearly, $T_{\varepsilon_N}^{(N,1)}$ and $T_{\varepsilon_N}^{(N,2)}$ are also stopping times with respect to $(\Fs_n)_{n=0}^N$, and consequently, $T_N$ as well.

By construction, the corresponding value process is given by
$$V_n(\hat{\phi}^N)=\frac{1}{N}\sum\limits_{i=1}^{4}\Ss_{n\wedge T_N}^{(i)}.$$
In particular, we have
\begin{equation}\label{fac2}
 V_n(\hat{\phi}^N)=V_{n\wedge T_N}(\phi^N)\geq -4\,\varepsilon_N + \frac{1}{N}\left(\theta_{T_N}^{(1)}+\theta_{T_N}^{(2)}\right)1_{\{n\geq T_N\}}.
\end{equation}
Next lemma provides a uniform control for the second term on the right-hand side of \eqref{fac2}.
\begin{lemma}\label{th12}
For $i=1,2,3$, there exists a constant $C_{i}^\theta>0$ such that
$$\sup\limits_{1\leq n\leq N}|\theta_{n}^{(i)}|\leq C_{i}^\theta\,N^{H-\frac{1}{2}}.$$
\end{lemma}
\begin{proof}
It follows from the definition of the random variables $\theta_{n}^{(i)}$ and Lemma \ref{eji}. 
\end{proof}
Now, motivated by our previous results, we choose
$$\varepsilon_N:=\frac{\ln(N)}{N^{\frac{1}{2}\wedge(2-2H)}}\quad\textrm{ and }\quad\hat{c}_N:=4\varepsilon_N+\frac{C_{1,2}}{N^{\frac{3}{2}-H}},$$
where $C_{1,2}=C_1^\theta+C_2^\theta$. 

Finally, for each $N\geq 1$, we define $\psi^N=(\psi^{0,N},\psi^{1,N})$ as follows. The position in stock, $\psi^{1,N}$, is given by:
\begin{equation}\label{str}
\psi_k^{1,N}:=\frac{1}{\sqrt{\hat{c}_N}}\hat{\phi}_k^{1,N},\quad k\in\{-1,0,...,N\},
\end{equation}
and the position in bond, $\psi^{0,N}$, is constructed as before, through the self-financing conditions \eqref{selffineq}, setting $\lambda_N=0$ and $\psi_{-1}^{0,N}=0$. 
\begin{theorem}
The sequence of self-financing strategies $(\psi^N)_{ N\geq 1}$ provides a strong asymptotic arbitrage in the large fractional binary market. 
\end{theorem}
\begin{proof}
In order to have a strong asymptotic arbitrage, we need to show that the two conditions of Definition~\ref{SAAd} are satisfied. 
More precisely, we prove that these two conditions are verified for 
$$c_N:=\sqrt{\hat{c}_N}\xrightarrow[N\rightarrow\infty]{} 0\quad\textrm{and}\quad C_N:=\frac{\vartheta}{2\sqrt{\hat{c}_N}}\xrightarrow[N\rightarrow\infty]{} \infty.$$
Note that, from Lemma \ref{th12} and equation \eqref{fac2}, the self-financing strategy $\hat{\phi}^{N}$ is $\hat{c}_N$-admissible. Since, in addition
$$V_k(\psi^{N})=\frac{1}{c_N}V_k(\hat{\phi}^{N}),\quad k\in\{0,...,N\},$$
we deduce that ${\psi}^{N}$ is $c_N$-admissible. Regarding the second condition, we use the convergence behaviour of $V_N(\phi^{N})$ given in Corollary~\ref{AA1conv}. First, note that $$\{T_{\varepsilon_N}^{(N,i)}-1\leq N\}=\{T_{\varepsilon_N}^{(N,i)}\leq N\},$$
and then, from the choice of $\varepsilon_N$ and Lemma \ref{mi123} and Corollary \ref{T4}, we obtain that
\begin{equation}\label{tntz}
 P(T_N\leq N)\leq\sum\limits_{i=1}^4 P\left(T_{\varepsilon_N}^{(N,i)}\leq N\right)\xrightarrow[N\rightarrow\infty]{} 0.
\end{equation}
On the other side, over the set $\{N<T_N\}$, we have
$$V_N(\psi^{N})=\frac{1}{c_N}V_N(\phi^{N}).$$
In  particular, we get
\begin{align*}
P\left(V_N(\phi^{N})>\vartheta/2\right)&=P\left(\{V_N(\phi^{N})>\vartheta/2\}\cap\{T_N\leq N\}\right)\\
&\qquad+P\left(\{V_N(\phi^{N})>\vartheta/2\}\cap\{T_N> N\}\right)\\
&\leq P\left(T_N\leq N\right)+P\left(V_N(\psi^{N})>C_N\right).
\end{align*}
Letting now $N\to\infty$ and applying the results of Corollary~\ref{AA1conv} and \eqref{tntz}, we get that 
$$\lim_{N\to\infty}P\left(V_N(\psi^{N})>C_N\right)=1.$$
The desired result is then proven.
\end{proof}
\section{Strong asymptotic arbitrage with transaction costs}\label{S4}

In this section, we let each $N$-fractional binary market be subject to $\la_N$ transaction costs, and we show that there exists a strong asymptotic arbitrage if the sequence of transaction costs $(\la_N)_{N\geq1}$ converges to zero fast enough. The corresponding sequence of self-financing strategies $(\psi^N(\lambda_N))_{N\geq 1}$ is constructed as follows. The position in stock is given by $\psi^{1,N}$, as in \eqref{str} of the frictionless case. The position in bond, $\psi^{0,N}(\lambda_N)$, is constructed from $\psi^{1,N}$ through the $\lambda_N$-self-financing conditions \eqref{selffineq}, setting $\psi_{-1}^{0,N}(\lambda_N):=0$.
\begin{theorem}
 The self-financing strategies  $(\psi^N(\lambda_N))_{N\geq 1}$, where
  $$\lambda_N=o\left(\frac{\sqrt{\ln{N}}}{N^{(2H-\frac14)\wedge(H+\frac12)}}\right),$$
 provide a strong asymptotic arbitrage in the large fractional binary markets with $(\la_N)_{N\geq 1}$ transaction costs.
\end{theorem}

\begin{proof}
In order to show the first condition of Definition~\ref{SAAd}, we have to make sure that the admissibility condition in the presence of transaction costs is fulfilled. Since $\psi_k^{1,N}=\frac{1}{c_N}\hat{\phi}_k^{1,N}$, we have that
\begin{equation}\label{vpr}
   V_n^{\la_N}(\psi^{N}(\lambda_N))=\frac{1}{c_N}\,V_n^{\la_N}(\hat{\phi}^{N}(\lambda_N)),
\end{equation}
where $\hat{\phi}^N(\lambda_N)=(\hat{\phi}^{0,N}(\lambda_N),\hat{\phi}^{1,N})$ and $\hat{\phi}^{0,N}(\lambda_N)$ is determined from $\hat{\phi}^{1,N}$ by means of the $\lambda_N$-self-financing conditions \eqref{selffineq}. Additionally, from \eqref{value} we deduce that
\begin{equation}\label{abc}
  V_n^{\la_N}(\hat{\phi}^{N}(\lambda_N))=V_0^{\la_N}(\hat{\phi}^{N}(\lambda_N))+V_n(\hat{\phi}^{N})-\lambda_N\left(\Vs_n^1+\Vs_n^2+\Vs_n^3\right),
\end{equation}
where
 \begin{align*}
\Vs_n^1&:=\sum\limits_{k=1}^{n}\ind_{\{\Delta_k\hat{\phi}^{1,N}\geq 0\}}\,\Delta_k\left[{(\hat{\phi}^{1,N})}^+ S^{N}\right],\\
\Vs_n^2&:=\sum\limits_{k=1}^{n}\ind_{\{\Delta_k\hat{\phi}^{1,N}< 0\}}\,\Delta_k\left[{(\hat{\phi}^{1,N})}^- S^{N}\right],\\
\Vs_n^3&:=\sum\limits_{k=1}^{n}\ind_{\{\Delta_k\hat{\phi}^{1,N}< 0\}}\,\hat{\phi}_{k-1}^{1,N}\Delta_k S^{N}.
 \end{align*} 
  Using \eqref{V0} and that $\phi_0^{1,N}=0$, we see that $V_0^{\la_N}(\hat{\phi}^{N}(\lambda_N))=0$. The second term in \eqref{abc} is exactly the value process with $0$ transaction costs for the trading strategy $\hat{\phi}^{N}$ and then, from the results of the previous section we have 
 $$V_n(\hat{\phi}^{N})\geq - \hat{c}_N.$$
For the third term, we proceed as follows. Using that $|\hat{\phi}_k^{1,N}|\leq |\phi_k^{1,N}|$, we obtain that
  \begin{align}\label{aa}
   |\Vs_n^1|&\leq \sum\limits_{k=1}^{n}|\phi^{1,N}_k|S_k^N+\sum\limits_{k=1}^{n}|\phi^{1,N}_{k-1}|S_{k-1}^N\nonumber\\
   &\leq \frac1{N^{1-H}}\left(\sum\limits_{k=1}^{n}|X_k|+\sum\limits_{k=2}^{n}|X_{k-1}|\right).
  \end{align}
  For the latter sums, we use the estimates in Lemma~\ref{eg1} and Lemma~\ref{eji} to obtain
  \begin{align}\label{Xk}
   \sum\limits_{k=1}^{n}|X_k|&\leq\sum\limits_{k=1}^{n}\sum_{l=1}^{k-1}j_k(l)+\sum\limits_{k=1}^{n}g_k\nonumber\\
   &\leq \sum_{k=1}^{n}\left(\frac{C_1}{k^{2-2H}}\sum_{l=1}^{\frac{k}4}\frac1{l^{H-\frac12}}+C_2\sum_{l=1}^{\frac{3k}4}\frac1{l^{\frac32-H}}\right)+2^{H-\frac12}g\,n\nonumber\\
   &\leq \tilde{C}\sum_{k=1}^n k^{H-\frac12}+2^{H-\frac12}g\,
   n\nonumber\\
   &\leq\hat{C}_1\, n^{H+\frac12},
  \end{align}
where $\tilde{C}$ and $\hat{C}_1$ are appropriate strictly positive constants. Similarly, we have
$$\sum\limits_{k=2}^{n}|X_{k-1}|\leq \hat{C}_1\,n^{H+\frac12}.$$
Hence, we deduce that
\begin{align}\label{Ca}
 |\Vs_n^1|\leq\hat{C}_1\frac{n^{H+\frac12}}{N^{1-H}}\leq \hat{C}_1 N^{2H-\frac12}.
\end{align}
For the term $\Vs_n^2$ in \eqref{abc}, we proceed in a similar way, and we obtain, for some constant $\hat{C}_2>0,$ that
\begin{align}\label{Cb}
|\Vs_n^2|\leq \hat{C}_2 N^{2H-\frac12}.
\end{align}
It is left to find an estimate for $\Vs_n^3$. Using \eqref{theta} we write
  \begin{align}\label{C1234}
  |\Vs_n^3|&\leq \sum\limits_{k=1}^{n}|\phi_{k-1}^{1,N}\Delta_k S^N|=\frac1{N}\sum\limits_{k=1}^{n}|X_{k-1}X_k|\nonumber\\
  &\leq\frac1{N}\sum_{k=1}^n\left|\theta_k^{(1)}+\theta_k^{(2)}+\theta_k^{(3)}\right|+\frac1{N}\sum_{k=1}^n\left|\theta_k^{(4)}\right|.
  \end{align}
From Lemma \ref{th12}, we have, for $C_{1,2,3}=C_1^\theta+C_2^\theta+C_3^\theta>0$, that
$$\sup\limits_{1\leq k\leq N}\left|\theta_k^{(1)}+\theta_k^{(2)}+\theta_k^{(3)}\right|\leq C_{1,2,3}\,N^{H-\frac{1}{2}}.$$
  We conclude that
  $$\frac{1}{N}\,\sum_{k=1}^n\left|\theta_k^{(1)}+\theta_k^{(2)}+\theta_k^{(3)}\right|\leq C_{1,2,3} N^{H-\frac{1}{2}}$$
 For the last term in \eqref{C1234}, we first notice that, using Lemma \ref{eji} and performing a similar calculation like in \eqref{Xk}, one gets
 that $\sum_{\ell=1}^{k-1}j_k(\ell)\leq \bar{C} k^{H-\frac12}$, for some constant $\bar{C}>0$. Using this and the definition of $\theta_k^{(4)}$, we obtain
 \begin{align*}
 \frac1{N}\sum_{k=1}^n\left|\theta_k^{(4)}\right|\leq \frac{\bar{C}^2}{N}\sum_{k=1}^n k^{2H-1}\leq \bar{C}^2 N^{2H-1}.
 \end{align*}
Hence, we derived that
\begin{align}\label{Cc}
 |\Vs_n^3|\leq \hat{C}_3 N^{2H-1},
\end{align}
for an appropriate constant $\hat{C}_3>0$.

From \eqref{abc} we deduce that:
\begin{align}\label{vphi}
 V_n^{\la_N}(\hat{\phi}^{N}(\lambda_N))\geq V_n(\hat{\phi}^{N}) -c_* \lambda_N\,N^{2H-\frac12},
\end{align}
for some constant $c_*>0$. 

We return to the self-financing trading strategy $\psi^N$. Thanks to \eqref{vpr}, we get
\begin{equation}\label{vapsi}
 V_n^{\la_N}(\psi^{N}(\lambda_N))\geq V_n(\psi^{N}) -c_*\frac{\lambda_N}{c_N} \,N^{2H-\frac12}.
\end{equation}
Since $\psi^N$ is $c_N$-admissible, we deduce that
\begin{equation*}
 V_n^{\la_N}(\psi^{N}(\lambda_N))\geq-c_N -c_*\frac{\lambda_N}{c_N} \,N^{2H-\frac12}=:-c_N(\lambda_N),
\end{equation*}
or equivalently, that $\psi^N(\lambda_N)$ is $c_N(\lambda_N)$-admissible. Note that, it is enough to choose 
$$\la_N=o\left(\frac{c_N}{N^{2H-\frac12}}\right)=o\left(\frac{\sqrt{\ln{N}}}{N^{(2H-\frac14)\wedge(H+\frac12})}\right),$$
to have $c_N(\lambda_N)\xrightarrow[N\rightarrow\infty]{}0$.

The second condition of Definition~\ref{SAAd} follows immediately. Indeed, defining 
$$C_N(\lambda_N):=C_N-c_*\frac{\lambda_N}{c_N} \,N^{2H-\frac12}\xrightarrow[N\rightarrow\infty]{}\infty,$$
and using \eqref{vapsi}, we obtain that
\begin{align*}
 P\left(V_N^{\la_N}(\psi^{N}(\lambda_N))\geq C_N(\lambda_N)\right)&\geq  P\left(V_N(\psi^{N})\geq C_N\right).\\
\end{align*}
The second condition follows from the properties of $(\psi^N)_{N\geq 1}$, and the desired result is proven.
\end{proof}

\bibliographystyle{plain}
\bibliography{reference}
\end{document}